\documentclass[11pt]{amsart}
\usepackage{amssymb,amsmath,amsthm}
\usepackage{tikz-cd}

\theoremstyle{plain}
\newtheorem{theorem}{Theorem}[section]
\newtheorem{lemma}[theorem]{Lemma}
\newtheorem{proposition}[theorem]{Proposition}

\theoremstyle{remark}
\newtheorem*{remark}{Remark}

\newcommand{\CC}{\mathbb{C}}
\newcommand{\DD}{\mathbb{D}}

\newcommand{\cB}{\mathcal{B}}
\renewcommand{\tilde}{\widetilde}

\let\Re\undefined
\let\Im\undefined
\DeclareMathOperator{\Re}{\mathrm{Re}}
\DeclareMathOperator{\Im}{\mathrm{Im}}

\DeclareMathOperator{\hol}{\mathrm{Hol}}
\DeclareMathOperator{\area}{\mathrm{area}}
\DeclareMathOperator{\diam}{\mathrm{diam}}
\DeclareMathOperator{\conv}{\mathrm{conv}}

\begin{document}

\title[On  the Crouzeix ratio for $N\times N$ matrices]{On  the Crouzeix ratio for $N\times N$ matrices}

\author[B. Malman]{Bartosz Malman}
\address{Division of Mathematics and Physics, M\"alardalen University, V\"aster\aa s, Sweden}
\email{bartosz.malman@mdu.se} 

\author[J. Mashreghi]{Javad Mashreghi}
\address{D\'epartement de math\'ematiques et de statistique, Universit\'e Laval, Qu\'ebec (QC), G1V 0A6, Canada}
\email{javad.mashreghi@mat.ulaval.ca}

\author[R. O'Loughlin]{Ryan O'Loughlin}
\address{D\'epartement de math\'ematiques et de statistique, Universit\'e Laval, Qu\'ebec (QC), G1V 0A6, Canada}
\email{ryan.oloughlin.1@ulaval.ca}

\author[T. Ransford]{Thomas Ransford}
\address{D\'epartement de math\'ematiques et de statistique, Universit\'e Laval, Qu\'ebec (QC), G1V 0A6, Canada}
\email{thomas.ransford@mat.ulaval.ca}

\date{21 Sep 2024}

\begin{abstract}
The Crouzeix ratio $\psi(A)$ of an $N\times N$ complex matrix $A$ is the supremum of $\|p(A)\|$
taken over all polynomials $p$ such that $|p|\le 1$ on the numerical range of $A$.
It is known that $\psi(A)\le 1+\sqrt{2}$, and it is conjectured that $\psi(A)\le 2$. 
In this note, we show that $\psi(A)\le C_N$, where $C_N$ is a constant depending only on $N$
and satisfying $C_N<1+\sqrt{2}$.
The proof is based on a study of the continuity properties of the map $A\mapsto \psi(A)$.
\end{abstract}

\thanks{Part of this work was carried out during Malman's visit to Universit\'e Laval, supported by 
the Simons--CRM Scholar-in-Residence program. Mashreghi's research was supported by an NSERC Discovery Grant and the Canada Research Chairs program. O'Loughlin was supported by a CRM--Laval Postdoctoral Fellowship. Ransford's research was supported by an NSERC Discovery Grant.}

\keywords{Operator, matrix, spectrum, numerical range, Crouzeix ratio}
	
\makeatletter
\@namedef{subjclassname@2020}{\textup{2020} Mathematics Subject Classification}
\makeatother
	
\subjclass[2020]{15A60, 47A12, 47A30}

\maketitle

\section{Introduction}\label{S:Intro}

Let $H$ be a 
complex Hilbert space, and let 
$\cB(H)$ be the algebra of bounded linear
operators on $H$, equipped with the operator norm.
Given $T\in\cB(H)$, we write $\sigma(T)$ for the spectrum of $T$, and
$W(T)$ for the \emph{numerical range} of $T$, namely the set
\[
W(T):=\bigl\{\langle Tx,x\rangle: x\in H, \|x\|=1\bigr\}.
\]
It is well known that $W(T)$ is a bounded convex subset of  $\CC$ whose closure contains $\sigma(T)$.
If further $\dim H<\infty$, then $W(T)$ is also compact.

The central object of study in this note is the \emph{Crouzeix ratio} of $T\in\cB(H)$, defined by
\begin{equation}\label{E:psidef}
\psi(T):=\sup\Bigl\{\|p(T)\|:p \text{~is a polynomial}, |p|\le 1 \text{~on~}W(T)\Bigr\}.
\end{equation}
It was first studied by Crouzeix in \cite{Cr04}.
The terminology `Crouzeix ratio' is taken from \cite{CGL24} and \cite{Ov22}.
We always have $\psi(T)\ge 1$ (consider $p\equiv1$),
with equality if $T$ is a normal operator. 
It is also easy to see that $\psi(U^*TU)=\psi(T)$
for all unitary operators $U$, and that
$\psi(\alpha T+\beta I)=\psi(T)$ for all $\alpha,\beta\in\CC$ with $\alpha\ne0$.

It is not obvious, \emph{a priori}, that the Crouzeix ratio is always finite.
That this is indeed the case was first proved by Delyon and Delyon in \cite{DD99}.
This is sometimes expressed by saying that $W(T)$ is  a \emph{$\psi(T)$-spectral set} for $T$.
As a consequence of their result, the homomorphism $p\mapsto p(T)$ extends by continuity
to a homomorphism $f\mapsto f(T)$ defined for all $f\in A(W)$, where $W=\overline{W(T)}$, 
and $A(W)$ is the uniform algebra of all continuous functions on $W$ that are holomorphic on $W^\circ$,
 the interior of $W$.
The extended map, often called the functional calculus for $T$,  satisfies
\begin{equation}\label{E:fc}
\|f(T)\|\le\psi(T)\sup_{W}|f| \quad(f\in A(W)).
\end{equation}

The same article also contains a result \cite[Theorem~3]{DD99}  
that implies a quantitative bound for $\psi(T)$, namely
\[
\psi(T)\le \Bigl(\frac{2\pi \diam(W)^2}{\area(W)}\Bigr)^3+3.
\]
However, this does not yield a universal numerical bound for $\psi(T)$. The first such bound was 
obtained by Crouzeix, who showed in \cite[Theorems~1 and~2]{Cr07} that we always have
\[
\psi(T)\le 11\!\cdot\!08.
\]
Some years later, Crouzeix and Palencia \cite[Theorem~3.1]{CP17} improved this estimate to
\begin{equation}\label{E:CPbound}
\psi(T)\le 1+\sqrt{2},
\end{equation}
at the same time greatly simplifying the proof.
Recently, this bound was further improved to 
\begin{equation}\label{E:a(W)}
\psi(T)\le 1+\sqrt{1+a(W)},
\end{equation}
where, once again $W=\overline{W(T)}$, and where $a(W)$ is the so-called analytic configuration
constant of $W$, which is a number depending only on $W$ 
and satisfying $0\le a(W)<1$ (see \cite[Theorem~2 and Proposition~26]{MMOR24}).
In particular, we always have
\begin{equation}\label{E:CPstrict}
\psi(T)<1+\sqrt{2},
\end{equation}
so the Crouzeix--Palencia bound  in \eqref{E:CPbound} is never attained.

To continue the discussion, it will be convenient to introduce a further piece of notation. 
For each $N\ge1$, we write $M_N(\CC)$ for the algebra of complex $N\times N$ matrices, and define
\begin{equation}\label{E:CNdef}
C_N:=\sup\{\psi(A): A \in M_N(\CC)\}.
\end{equation}
It is easy to see that $C_1=1$ and that $C_N\le C_{N+1}$ for all $N\ge1$.
Clearly, from \eqref{E:CPbound}, we have $C_N\le 1+\sqrt{2}$ for all $N$, 
so $C_N$ converges to a limit $C$, say, where $C\le 1+\sqrt{2}$. 
An approximation argument (see e.g.\ \cite[Theorem~2]{Cr07})
then shows that $\psi(T)\le C$ for all Hilbert-space operators $T$. 
Thus there is some interest
in determining or estimating the values of $C_N$.

Crouzeix showed in \cite[Theorem~1.1]{Cr04} that $C_2=2$.
He did this by finding an explicit formula for $\psi(A)$ when
$A\in M_2(\CC)$.
He further conjectured that $C_N=2$ for all $N\ge2$.
This conjecture remains open.
There is a substantial amount of supporting numerical evidence (see e.g.\ \cite{GO18,Ov22}), and the conjecture is known to be true for many special classes
of matrices (see e.g.\ \cite{BCD06,BG24,CGL18,Ch13,CGL24,GKL18,GC12}).
However, even for $3\times 3$ matrices, no universal bound for the Crouzeix ratio seems to be known, beyond
the estimate $C_3\le 1+\sqrt{2}$ already mentioned above. The following result,
which is our main theorem, may therefore be of  interest.

\begin{theorem}\label{T:main}
For each $N\ge 1$, we have the strict inequality $C_N<1+\sqrt{2}$.
\end{theorem}

In view of the pointwise estimate \eqref{E:CPstrict},
a natural approach to  Theorem~\ref{T:main}
is to show that the supremum in \eqref{E:CNdef}
is always attained. This suggests trying some form
of compactness argument. However,
a compactness argument
presupposes the continuity of the map $A\mapsto\psi(A)$,
and it is not hard to see that this map is discontinuous
at every multiple of the identity. Worse still, $\psi$ has other
discontinuities as well, which turns out to be  a more serious 
obstacle.
Fortunately,
however, $\psi$ is continuous at a large enough set of points
for us to be able to push through the  
compactness argument, as proposed.

Here, in more detail, is a plan of the article.
In \S\ref{S:continuity}, we investigate the continuity
properties of the map $A\mapsto\psi(A)$ on $M_N(\CC)$.
We show that it is lower semicontinuous everywhere, continuous in some places and discontinuous at others. The proof of 
Theorem~\ref{T:main} is presented in \S\ref{S:proof}.
Finally, in \S\ref{S:conclusion}, we make some concluding
remarks and pose some questions.


\section{Continuity properties of the Crouzeix ratio}\label{S:continuity}

\subsection{Lower semicontinuity}

Our first result shows that $\psi$ is lower semicontinuous
on $M_N(\CC)$.

\begin{theorem}\label{T:lsc}
If $A_n\to A$ in $M_N(\CC)$,
then
\begin{equation}\label{E:lsc}
\liminf_{n\to\infty}\psi(A_n)\ge \psi(A).
\end{equation}
\end{theorem}

To prove this result, it will be convenient to introduce an auxiliary notion.
Given $A\in M_N(\CC)$ and an open neighbourhood $U$ of $\sigma(A)$, we define the 
\emph{relative Crouzeix ratio} of the pair $(A,U)$ by
\[
\psi_U(A):=\sup\{\|f(A)\|:f\in H^\infty(U),\,|f|\le 1\text{~on~}U\}.
\]
This quantity was  considered (under another name) by Crouzeix in \cite{Cr04}. 
Among other results, he showed in \cite[Lemma~2.2]{Cr04} that
\begin{equation}\label{E:2psis}
\psi(A)=\sup_{U\supset W(A)}\psi_U(A),
\end{equation}
where the supremum is taken over all open neighbourhoods $U$ of $W(A)$.

\begin{proof}[Proof of Theorem~\ref{T:lsc}]
Let $A_n\to A$ in $M_N(\CC)$.
Let $U$ be an open neighbourhood of $W(A)$,
and let $f\in H^\infty(U)$ with $|f|\le1$ on $U$.
For all large enough~$n$, 
we have $W(A_n)\subset U$, 
and in particular $\sigma(A_n)\subset U$.
For each such $n$, we clearly have $\|f(A_n)\|\le \psi_U(A_n)\le \psi(A_n)$.
Also we have $f(A_{n})\to f(A)$ as $n\to\infty$, 
and in particular  $\|f(A_{n})\|\to\|f(A)\|$.
Hence
\[
\liminf_{n\to\infty}\psi(A_n)\ge \liminf_{n\to\infty}\|f(A_n)\|=\|f(A)\|.
\]
Taking the supremum of the right-hand side 
over all $f\in H^\infty(U)$ such that $|f|\le 1$ on $U$, we deduce that
\[
\liminf_{n\to\infty}\psi(A_n)\ge \psi_U(A).
\]
Finally, taking the supremum of the right-hand side over all open neighbourhoods $U$ of $W(A)$
and using \eqref{E:2psis}, we obtain the desired conclusion \eqref{E:lsc}.
\end{proof}


\subsection{Upper semicontinuity}

The following result shows that $\psi$ is upper
semicontinuous (and hence continuous) at each
matrix whose spectrum is contained in the interior
of its numerical range.

\begin{theorem}\label{T:usc}
Let  $A\in M_N(\CC)$ be such that $\sigma(A)\subset W(A)^\circ$.
 If $A_n\to A$, then
\[
\limsup_{n\to\infty}\psi(A_n)\le \psi(A).
\]
\end{theorem}

\begin{proof}[Proof of Theorem~\ref{T:usc}]
Suppose for a contradiction, that 
there exist $\epsilon>0$ and a sequence $(A_n)$ such that $A_n\to A$ in $M_N(\CC)$ and
\begin{equation}\label{E:contradiction}
\psi(A_n)>\psi(A)+\epsilon \quad(n\ge1).
\end{equation}
Then there exist polynomials $p_n$ such that, for all $n\ge1$, 
\[
\sup_{W(A_n)}|p_n|\le1 
\quad\text{and}\quad
\|p_n(A_n)\|>\psi(A)+\epsilon.
\]
Since $A_n\to A$, every compact subset of $W(A)^\circ$
will eventually be contained in $W(A_n)^\circ$.
Thus, by a normal-family argument, a subsequence of the $p_n$ 
(which, by relabelling, we may suppose to be the whole
sequence) converges locally uniformly on $W(A)^\circ$
to a holomorphic function $f$  such that $|f|\le1$ on $W(A)^\circ$.
Since $A_n\to A$ and $\sigma(A)\subset W(A)^\circ$, we have $p_n(A_n)\to f(A)$, and in particular 
$\|p_n(A_n)\|\to\|f(A)\|$. Thus $\|f(A)\|\ge \psi(A)+\epsilon$.

As the whole situation is invariant under translation
in the complex plane, there is no loss of generality in supposing,
from the outset,
that $0\in W(A)^\circ$. For $r\in(0,1)$, let $f_r$ denote the $r$-dilation of $f$, given by $f_r(z):=f(rz)$. 
Note that $f_r\in A(W(A))$ and that $|f_r|\le 1$ on $W(A)$.
Since $f_r\to f$ as $r\to1^-$
uniformly on a neighbourhood of $\sigma(A)$, 
it follows that $\|f_r(A)\|\to\|f(A)\|$.
In particular, if $r$ is sufficiently close to $1$,
then $\|f_r(A)\|>\psi(A)$.
This contradicts \eqref{E:fc}.
\end{proof}

\subsection{Discontinuity}

It is easy to see that,
if $N\ge2$, then $A\mapsto\psi(A)$ is discontinuous at $A=0$.
Indeed, since $\psi(\alpha A)=\psi(A)$ for all $\alpha\ne0$,
we have 
\[
\limsup_{A\to 0}\psi(A)=\sup_{A\in M_N(\CC)}\psi(A)\ge2,
\]
whereas $\psi(0)=1$.
Since $\psi(A)=\psi(A+\beta I)$ for all $\beta\in\CC$, 
it follows that $\psi$ is discontinuous at every multiple of the identity matrix.

If $N\ge3$, then $\psi$ is also discontinuous 
at some matrices $A$ that are not multiples of the identity.
The following result gives an example of this sort of discontinuity.

\begin{proposition}\label{P:discont}
For $\alpha\ge0$, let 
\[
A_\alpha:=
\begin{pmatrix} 
1&0&0\\ 0&0&\alpha\\0&0&0
\end{pmatrix}.
\]
Then $\psi(A_\alpha)\ge\pi/2$ for all $\alpha>0$ and
$\psi(A_0)=1$. Consequently $\psi$ is discontinuous at $A_0$.
\end{proposition}

\begin{proof}
Clearly $A_0$ is self-adjoint, so $\psi(A_0)=1$.

Now let $\alpha>0$.
Then 
\[W(A_\alpha)=\conv\Bigl(\{z:|z|\le\alpha/2\}\cup\{1\}\Bigr),
\]
where $\conv(\cdot)$ denotes the convex hull.
In particular, 
\[
W(A_\alpha)\subset S_\alpha:=\{z\in\CC:|\Im z|\le \alpha/2\}.
\]
Consider the function $f(z):=\tanh(\pi z/(2\alpha))$,
which is a conformal mapping of $S_\alpha$ onto
the closed unit disk $\overline{\DD}$.
We have $|f(z)|\le1$ for all $z\in W(A_\alpha)$. 
Also, substituting $A_\alpha$ directly into the 
Taylor series expansion of $f(z)$, we see that
\[
f(A_\alpha)=
\begin{pmatrix} 
f(1)&0&0\\ 0&0&\alpha f'(0)\\0&0&0
\end{pmatrix}
=
\begin{pmatrix} 
\tanh(\pi/(2\alpha))&0&0\\ 0&0&\pi/2\\0&0&0
\end{pmatrix}.
\]
In particular $\|f(A_\alpha)\|\ge\pi/2$. It follows that
$\psi(A_\alpha)\ge\pi/2$, as claimed.
\end{proof}


\section{Proof of Theorem~\ref{T:main}}\label{S:proof}

As mentioned in the Introduction, the strategy for proving Theorem~\ref{T:main}
is to deduce it from the pointwise estimate \eqref{E:CPstrict}
using a form of compactness argument. This endeavour is complicated by the fact that, as we 
have just seen, the map $A\mapsto\psi(A)$ is discontinuous at certain points.
The discontinuity at multiples of the identity is not a problem, since,
using the relation $\psi(A+\beta I)=\psi(A)$, we may consider $\psi$
as being defined on the quotient space $M_N(\CC)/ \CC I$,
and  work on that space. The multiples of the identity then
`disappear'. On the other hand, the discontinuity of $\psi$ at points other than multiples of the identity,
such as the one observed in Proposition~\ref{P:discont}, is more problematic.
We deal with these by using a decomposition  technique to reduce the dimension.
The final proof is therefore a combination of a compactness argument and an induction.

The reduction of dimension is the subject of
the following theorem.
We recall that the constants $C_N$ were defined in \eqref{E:CNdef}.

\begin{theorem}\label{T:induction}
Let $A\in M_N(\CC)\setminus\CC I$ be such that $\sigma(A)\cap\partial W(A)\ne\emptyset$. 
If $A_n\to A$, then
\[
\limsup_{n\to\infty}\psi(A_n)\le C_{N-1}.
\]
\end{theorem}

To prove this theorem, we require two lemmas. The first of these 
is a decomposition result.

\begin{lemma}\label{L:decomp}
Let  $A\in M_N(\CC)$.
Then $A$ is unitarily equivalent to a block matrix $D\oplus \tilde{A}$,
where $D$ is diagonal, $\sigma(D)\subset\partial W(A)$ 
and $\sigma(\tilde{A})\subset W(A)^\circ$.
\end{lemma}

\begin{remark}
It is not excluded that $D$ or $\tilde{A}$ have dimension $0$.
\end{remark}

\begin{proof}
By Schur's theorem, we can suppose that $A$ is lower triangular.
The diagonal entries of $A$ are the eigenvalues $\lambda_1,\dots,\lambda_N$.
We can suppose that they are ordered so that $\lambda_1,\dots,\lambda_k\in\partial W(A)$
and $\lambda_{k+1},\dots,\lambda_N\in W(A)^\circ$, where $0\le k\le N$.

It remains to show that the off-diagonal entries in rows $1$ to $k$ are all zero. 
Let $e_1,\dots,e_N$ be the standard unit vector basis of $\CC^N$.
Fix $j$ with $1\le j\le k$ and let $m\in\{j+1,\dots,N\}$. 
We need to show that $\langle Ae_j,e_m\rangle=0$.
For each $\alpha\in\CC$, we have $\|e_j+\alpha e_m\|^2=1+|\alpha|^2$, so,
by definition of numerical range,
\[
\frac{\bigl\langle (A(e_j+\alpha e_m),\,(e_j+\alpha e_m)\bigr\rangle}{1+|\alpha|^2}\in W(A).
\]
Expanding out the left-hand side and using the facts that $\langle Ae_j,e_j\rangle=\lambda_j$
and $\langle Ae_m,e_j\rangle=0$, we obtain
\[
\lambda_j+\overline{\alpha}\langle Ae_j,e_m\rangle +O(|\alpha|^2)\in W(A)
\quad(\alpha\in\CC,\alpha\to0).
\]
Since $\lambda_j\in\partial W(A)$, this forces $\langle Ae_j,e_m\rangle=0$, as required.
\end{proof}

The second lemma is a general operator-theory result 
about the stability of invariant-subspace decompositions.

\begin{lemma}\label{L:invsub}
Let $H$ be a Hilbert space and
let $T\in \cB(H)$. Suppose that $H=X\oplus Y$, where $X,Y$ are (not necessarily
orthogonal) $T$-invariant subspaces
such that $\sigma(T|_X)\cap\sigma(T|_Y)=\emptyset$.
If $T_n\to T$ in $\cB(H)$, then there exists a sequence
$S_n\to I$ in $\cB(H)$ such that, for all sufficiently large $n$,
we have $S_n^{-1}T_nS_n(X)\subset X$
and $S_n^{-1}T_nS_n(Y)\subset Y$.
\end{lemma}

\begin{proof}
Fix disjoint open subsets $U,V$ of $\CC$ such that $\sigma(T|_X)\subset U$ and $\sigma(T|_Y)\subset V$. Then, for all large enough $n$, we have $\sigma(T_n)\subset U\cup V$.
Let $P_n,Q_n$
be the spectral projections of $T_n$ corresponding to $U,V$
respectively.
Then $P_n\to P$ and $Q_n\to Q$, where $P,Q$
are projections of $H$ onto $X,Y$ respectively such that $P+Q=I$.
Also  $P_nQ_n=Q_nP_n=0$ for all $n$.

Define $S_n:=P_nP+Q_nQ$. Then $S_n\to P^2+Q^2=P+Q=I$.
In particular, $S_n$ is invertible for all sufficiently large $n$.
For these $n$, we have $S_n(H)=H$, and so
\begin{align*}
P_n(H)&=P_nS_n(H)=P_n(P_nP+Q_nQ)(H)\\
&=P_n^2(X)+P_nQ_n(Y)=P_n(X)=S_n(X).
\end{align*}
As $P_n(H)$ is $T_n$-invariant, we deduce that
 $T_nS_n(X)\subset S_n(X)$, from which it follows that
$S_n^{-1}T_n S_n(X)\subset X$. Likewise $S_n^{-1}T_n S_n(Y)\subset Y$.
\end{proof}

Now we  return to Theorem~\ref{T:induction}

\begin{proof}[Proof of Theorem~\ref{T:induction}]
By assumption $A\notin \CC I$ and $\sigma(A)\cap\partial W(A)\ne\emptyset$.
Using Lemma~\ref{L:decomp}, we deduce that
$A$ is unitarily equivalent to a block matrix $E\oplus F$
with $E\in M_{N_1}(\CC)$ and $F\in M_{N_2}(\CC)$,
where $N_1,N_2\le N-1$ and
$\sigma(E)\cap \sigma(F)=\emptyset$.
We can suppose without loss of generality that $A=E\oplus F$.

Let $A_n\to A$ in $M_N(\CC)$.
By Lemma~\ref{L:invsub}, applied with
 $X:=\CC^{N_1}\oplus 0$ and $Y:=0\oplus \CC^{N_2}$,
there exists a sequence $S_n\to I$ in $M_N(\CC)$ such that,
for all sufficiently large $n$, we have
$S_n^{-1}A_nS_n=E_n\oplus F_n$, also a block matrix. 
Henceforth, we restrict attention to these $n$.

Let $p$ be a polynomial such that $|p|\le 1$ on $W(A_n)$. Then
\begin{align*}
\|p(A_n)\|&=\|S_np(E_n\oplus F_n)S_n^{-1}\| \\
&\le \|S_n\|\|S_n^{-1}\|\|p(E_n\oplus F_n)\|\\
&=\|S_n\|\|S_n^{-1}\|\max\{\|p(E_n)\|,\|p(F_n)\|\}
\end{align*}
Now $S_n(X)$ is $A_n$-invariant and $S_n|_X$ is an invertible map of $X$ onto $S_n(X)$.
We further have
\[
E_n=(S_n|_X)^{-1}\circ(A_n|_{S_n(X)})\circ(S_n|_X).
\]
Therefore
\[
\|p(E_n)\|\le \|(S_n|_X)^{-1}\|\|p(A_n|_{S_n(X)})\|\|(S_n|_X)\|\le \|S_n\|\|S_n^{-1}\|\|p(A_n|_{S_n(X)})\|.
\]
As $W(A_n|_{S_n(X)})\subset W(A_n)$, we have $|p|\le 1$ on $W(A_n|_{S_n(X)})$.
Furthermore $\dim S_n(X)=N_1\le N-1$. So, by the definition of $C_{N_1}$, it follows that
\[
\|p(A_n|_{S_n(X)})\|\le C_{N_1}\le C_{N-1}.
\]
We thus obtain that
\[
\|p(E_n)\|\le\|S_n\|\|S_n^{-1}\|C_{N-1}.
\]
Likewise for $p(F_n)$. Substituting these estimates into the bound for $\|p(A_n)\|$ above, we obtain
\[
\|p(A_n)\|\le \|S_n\|^2\|S_n^{-1}\|^2C_{N-1}.
\]
As this holds for all polynomials $p$ with $|p|\le 1$ on $W(A_n)$, we deduce that
\[
\psi(A_n)\le \|S_n\|^2\|S_n^{-1}\|^2C_{N-1}.
\]
Letting $n\to\infty$, we obtain
\[
\limsup_{n\to\infty}\psi(A_n)\le \limsup_{n\to\infty}\|S_n\|^2\|S_n^{-1}\|^2C_{N-1}=C_{N-1},
\]
the last equality because $S_n\to I$. The theorem is proved.
\end{proof}

Finally, we can complete the proof of Theorem~\ref{T:main}

\begin{proof}[Proof of Theorem~\ref{T:main}]
The proof is by induction on $N$. The result is trivial if $N=1$, since $C_1=1$.
Suppose now that $N\ge 2$ and that $C_{N-1}<1+\sqrt{2}$. 

We first claim that the function 
\[
\tilde{\psi}(A):=\max\{\psi(A),C_{N-1}\}
\]
is upper semicontinuous on $M_N(\CC)\setminus \CC I$ (and thus continuous there).
Indeed, let $A\in M_N(\CC)\setminus\CC I$, and let $A_n\to A$. 
If $\sigma(A)\subset W(A)^\circ$, then by Theorem~\ref{T:usc} we have
\[
\limsup_{n\to\infty} \psi(A_n)\le \psi(A).
\]
If, on the other hand, $\sigma(A)\not\subset W(A)^\circ$, then by Theorem~\ref{T:induction} we have
\[
\limsup_{n\to\infty}\psi(A_n)\le C_{N-1}.
\]
Either way, we have
\[
\limsup_{n\to\infty}\tilde{\psi}(A_n)\le \tilde{\psi}(A),
\]
justifying the claim.

As mentioned at the beginning of this section, the fact that $\psi(A+\beta I)=\psi(A)$ for all $\beta\in\CC$
allows us to view $\psi$ as a function defined on the quotient space $M_N(\CC)/\CC I$.
The same is therefore true of $\tilde{\psi}$, 
and, in addition, $\tilde{\psi}$ is continuous with respect to
the quotient norm, except at $0$. As $M_N(\CC)/\CC I$ is finite-dimensional, 
its unit sphere is compact, so $\tilde{\psi}$ attains a maximum there,
say at $A_0$. 

For $A\in M_N(\CC)$ and $\alpha\in\CC\setminus\{0\}$, 
we have $\psi(\alpha A)=\psi(A)$, and hence also
$\tilde{\psi}(\alpha A)=\tilde{\psi}(A)$.
Thus in fact $\tilde{\psi}$ attains a global maximum at $A_0$, that is,
\[
\max\{\psi(A),C_{N-1}\}\le \max\{\psi(A_0),C_{N-1}\}\quad\forall A\in M_N(\CC).
\]
Consequently 
\[
C_N\le\max\{\psi(A_0),C_{N-1}\}.
\]
Finally, we know that $\psi(A_0)<1+\sqrt{2}$
by the pointwise estimate \eqref{E:CPstrict}, 
and $C_{N-1}<1+\sqrt{2}$ by the inductive hypothesis.
We may therefore conclude that $C_N<1+\sqrt{2}$, 
thereby completing the induction.
\end{proof}


\section{Concluding remarks and questions}\label{S:conclusion}

(1) Theorem~\ref{T:usc} can be generalized as follows.

\begin{theorem}
Let  $A\in M_N(\CC)$. 
Suppose that $\sigma(A)\cap \partial W(A)$ 
is either empty or consists exclusively of 
simple eigenvalues of $A$. If $A_n\to A$ in $\cB(H)$, then
\[
\limsup_{n\to\infty}\psi(A_n)\le \psi(A).
\]
\end{theorem}

The proof is  by combining the techniques used to prove 
Theorems~\ref{T:usc} and~\ref{T:induction}. We omit the details, since the result is
not needed here. However, it does show that the only $3\times 3$ matrices
at which $\psi$ can be discontinuous are those unitarily equivalent to
matrices of the form $\alpha A_0+\beta I$,
where $A_0$ is the matrix in Proposition~\ref{P:discont}. This explains the choice of $A_0$
in that example.

(2) The constant $\pi/2$ appearing in Proposition~\ref{P:discont} is not optimal.
A more careful analysis shows that,
in the notation of Proposition~\ref{P:discont}, we  have
\[
\psi(A_\alpha)\ge\sup\{|f'(0)|: f\in\hol(D,\DD)\} 
\quad(\alpha>0),
\]
where $\DD$ is the open unit disk and $D:=\DD\cup\{z:\Re z>0,|\Im z|< 1\}$.
By considering $f(1/z)$, we deduce that
\[
\psi(A_\alpha)\ge \gamma(K) \quad(\alpha>0),
\]
where $K$ is the compact set formed by taking 
the complement of the image of $D$ under the map $z\mapsto 1/z$,
and $\gamma(K)$ denotes the analytic capacity of $K$.
A simple calculation shows that $K$ is
the union of three semi-disks as shown in  Figure~\ref{F:K}.
In particular, since $K$ is connected, its analytic capacity $\gamma(K)$
coincides with its logarithmic capacity $c(K)$. Thus
\[
\psi(A_\alpha)\ge c(K) \quad(\alpha>0).
\] 
It is plausible that $\limsup_{A\to A_0}\psi(A)=c(K)$, though we cannot prove it.
Note, however, that by Theorem~\ref{T:induction} we must have $\limsup_{A\to A_0}\psi(A)\le2$.

\begin{figure}[htb]
\begin{center}
\begin{tikzpicture}[scale=1.4]
\fill[fill=gray!40] (0,0) circle(1);
\fill[fill=white] (0,1)--(0,-1)--(1,-1)--(1,1)--(0,1);
\fill[fill=gray!40] (0,1/2) circle(1/2);
\fill[fill=gray!40] (0,-1/2) circle(1/2);
\filldraw (0,0) circle (1/2pt) node[left] {$0$};
\filldraw (-1,0) circle (1/2pt) node[left] {$-1$};
\filldraw (0,1) circle (1/2pt) node[above] {$i$};
\filldraw (0,-1) circle (1/2pt) node[below] {$-i$};
\end{tikzpicture}
\end{center}
\caption{The set $K$}\label{F:K}
\end{figure}
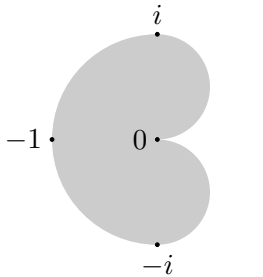

(3) Lemma~\ref{L:decomp}  implies a stronger form of itself, in which the inclusion 
$\sigma(\tilde{A})\subset W(A)^\circ$ is replaced by $\sigma(\tilde{A})\subset W(\tilde{A})^\circ$.
To see this, it suffices to reapply the lemma with $A$ replaced by $\tilde{A}$, 
and  repeat as often as necessary.

(4) It would be interesting to quantify the arguments used in proving that $C_N<1+\sqrt{2}$
to obtain concrete numerical estimates better than $1+\sqrt{2}$.
Even an estimate for $C_3$ would be of interest.

(5) Of course, the biggest problem is to identify $C:=\lim_{N\to\infty}C_N$.
According to Crouzeix's conjecture $C=2$. Can one at least show that $C<1+\sqrt{2}$?

\subsection*{Acknowledgement}
We thank \L ukasz Kosi\'nski for helpful discussions.

\bibliographystyle{plain}
\bibliography{crouzeix.bib}

\begin{thebibliography}{10}

\bibitem{BCD06}
C.~Badea, M.~Crouzeix, and B.~Delyon.
\newblock Convex domains and {K}-spectral sets.
\newblock {\em Math. Z.}, 252(2):345--365, 2006.

\bibitem{BG24}
K.~Bickel and P.~Gorkin.
\newblock Blaschke products, level sets, and {C}rouzeix's conjecture.
\newblock {\em J. Anal. Math.}, 152(1):217--254, 2024.

\bibitem{CGL18}
T.~Caldwell, A.~Greenbaum, and K.~Li.
\newblock Some extensions of the {C}rouzeix-{P}alencia result.
\newblock {\em SIAM J. Matrix Anal. Appl.}, 39(2):769--780, 2018.

\bibitem{Ch13}
D.~Choi.
\newblock A proof of {C}rouzeix's conjecture for a class of matrices.
\newblock {\em Linear Algebra Appl.}, 438(8):3247--3257, 2013.

\bibitem{Cr04}
M.~Crouzeix.
\newblock Bounds for analytical functions of matrices.
\newblock {\em Integral Equations Oper. Theory}, 48(4):461--477, 2004.

\bibitem{Cr07}
M.~Crouzeix.
\newblock Numerical range and functional calculus in {Hilbert} space.
\newblock {\em J. Funct. Anal.}, 244(2):668--690, 2007.

\bibitem{CGL24}
M.~Crouzeix, A.~Greenbaum, and K.~Li.
\newblock Numerical bounds on the {C}rouzeix ratio for a class of matrices.
\newblock {\em Calcolo}, 61(2):Paper No. 32, 18, 2024.

\bibitem{CP17}
M.~Crouzeix and C.~Palencia.
\newblock The numerical range is a {{\((1+\sqrt{2})\)}}-spectral set.
\newblock {\em SIAM J. Matrix Anal. Appl.}, 38(2):649--655, 2017.

\bibitem{DD99}
B.~Delyon and F.~Delyon.
\newblock Generalization of von {Neumann}'s spectral sets and integral
  representation of operators.
\newblock {\em Bull. Soc. Math. Fr.}, 127(1):25--41, 1999.

\bibitem{GKL18}
C.~Glader, M.~Kurula, and M.~Lindstr\"om.
\newblock Crouzeix's conjecture holds for tridiagonal {$3\times 3$} matrices
  with elliptic numerical range centered at an eigenvalue.
\newblock {\em SIAM J. Matrix Anal. Appl.}, 39(1):346--364, 2018.

\bibitem{GC12}
A.~Greenbaum and D.~Choi.
\newblock Crouzeix's conjecture and perturbed {J}ordan blocks.
\newblock {\em Linear Algebra Appl.}, 436(7):2342--2352, 2012.

\bibitem{GO18}
A.~Greenbaum and M.~L. Overton.
\newblock Numerical investigation of {C}rouzeix's conjecture.
\newblock {\em Linear Algebra Appl.}, 542:225--245, 2018.

\bibitem{MMOR24}
B.~Malman, J.~Mashreghi, R.~O'Loughlin, and T.~Ransford.
\newblock Double-layer potentials, configuration constants and applications to
  numerical ranges.
\newblock Preprint, arXiv:2407.19049.

\bibitem{Ov22}
M.~L. Overton.
\newblock Local minimizers of the {C}rouzeix ratio: a nonsmooth optimization
  case study.
\newblock {\em Calcolo}, 59(1):Paper No. 8, 19, 2022.

\end{thebibliography}

\end{document}